\documentclass[11pt]{amsart}
\usepackage{amssymb}

\usepackage{graphicx}
\usepackage{color}
\usepackage[cmtip,all]{xy}
\usepackage{url}
\usepackage{enumitem,kantlipsum}
\usepackage{comment}
\usepackage{float}
\usepackage{array}
\usepackage{aliascnt}
\usepackage{hyperref}
\hypersetup{colorlinks=true,
linkcolor=blue,
filecolor=blue,
urlcolor=blue,
citecolor=blue}

\newtheorem{theorem}{Theorem}[section]
\newtheorem{claim}{Claim}

\theoremstyle{lemma}
\newaliascnt{lemma}{theorem}

\aliascntresetthe{lemma}

\theoremstyle{prop}
\newaliascnt{prop}{theorem}
\newtheorem{prop}[prop]{Proposition}
\aliascntresetthe{prop}

\theoremstyle{cor}
\newaliascnt{cor}{theorem}
\newtheorem{cor}[cor]{Corollary}
\aliascntresetthe{cor}

\theoremstyle{definition}
\newaliascnt{definition}{theorem}
\newtheorem{definition}[definition]{Definition}
\aliascntresetthe{definition}

\theoremstyle{remark}
\newaliascnt{remark}{theorem}
\newtheorem{remark}[remark]{Remark}
\aliascntresetthe{remark}

\theoremstyle{example}
\newaliascnt{example}{theorem}

\aliascntresetthe{example}

\numberwithin{equation}{section}

\begin{document}

\title[Explicit formulae for the Aicardi-Juyumaya bracket of tied links]{Explicit formulae for the Aicardi-Juyumaya bracket of tied links}


\author[O. Cárdenas-Andaur]{O'Bryan Cárdenas-Andaur}
\address{Departamento de Álgebra, Facultad de Matemáticas, Instituto de Matemáticas (IMUS), Universidad de Sevilla, Av. Reina Mercedes s/n, 41012 Sevilla, Spain}
\curraddr{}
\email{obryan.cardenas@uv.cl}
\author[J. González-Meneses]{Juan González-Meneses}
\address{Departamento de Álgebra, Facultad de Matemáticas, Instituto de Matemáticas (IMUS), Universidad de Sevilla, Av. Reina Mercedes s/n, 41012 Sevilla, Spain}
\curraddr{}
\email{meneses@us.es}
\thanks{}
\author[M. Silvero]{Marithania Silvero}
\address{Departamento de Álgebra, Facultad de Matemáticas, Instituto de Matemáticas (IMUS), Universidad de Sevilla, Av. Reina Mercedes s/n, 41012 Sevilla, Spain}
\curraddr{}
\email{marithania@us.es}
\thanks{}


\date{\today}

\dedicatory{}

\begin{abstract}
The double bracket $\langle \langle \cdot \rangle \rangle$ (also known as the AJ-bracket) is an invariant of framed tied links that extends the Kauffman bracket of classical links. Unlike the classical setting, little is known about the structure of AJ-states (analogous to classical Kauffman states) of a given tied link diagram, and no general state-sum formula for the AJ-bracket is currently available. In this paper we analyze the AJ-states of $2$- and $3$-tied link diagrams, and provide a complete description of their associated resolution trees leading to a computation of $\langle \langle \cdot \rangle \rangle$. As a result, we derive explicit state-sum formulas for the AJ-bracket. 
These are the first closed-form expressions of this kind, and they constitute a concrete step toward a combinatorial categorification of the tied Jones polynomial.
\end{abstract}

\maketitle
    
\section{Introduction}

Tied links were introduced in \cite{Aicardi2016} as a generalization of classical links where components are partitioned into different sets. This naturally raises the question of whether a given property of classical links can be extended to the setting of tied links. 

When considering the generalization of link invariants, one is led to ask whether the tied version (if any) of a given invariant is stronger than its classical counterpart. This is indeed the case for the tied Jones polynomial introduced in \cite{Aicardi2018}, which is able to distinguish tied links whose associated classical links (i.e., those obtained when forgetting the partitions) are not distinguished by the classical Jones polynomial. Some examples of pairs of links with the above property can be found in \cite{Cardenas2024}. 

As in the classical case, the tied version of the Kauffman bracket (the so-called AJ-bracket) plays a central role in the definition of the tied Jones polynomial. However, unlike the classical situation \cite{Kauffman1987}, no state-sum formula is known for this invariant, nor is the structure of the diagrams that play the role of Kauffman states (called AJ-states) fully understood. The main difficulty lies in the fact that one of the defining skein relations of the AJ-bracket alters the partitions of the components. As a consequence, when constructing a resolution tree to compute the AJ-bracket of a tied link diagram $D$, the order in which crossings are smoothed is crucial in determining which diagrams (AJ-states) appear as leaves of the tree. Thus, it may happen that an AJ-state is a leaf in a resolution tree $T_D$ of $D$, but does not appear in another resolution tree $T’_D$ of the same diagram. Even more, the number of leaves in $T_D$ and $T’_D$ may, in general, differ (see Figure 11 in \cite{Cardenas2024} for such an example). 

These difficulties prevent the formulation of state-sum expressions analogous to those of the Kauffman bracket in the general tied case.

In this paper, we study the structure of AJ-states of tied link diagrams, focusing on the cases of $2$-tied and $3$-tied links, that is, links whose components are partitioned into two and three sets, respectively. In particular, we compute the number of AJ-states and the number of leaves in any resolution tree of $2$-tied link diagrams, proving that their number does not depend on the chosen order of smoothings. We also compute the number of AJ-states and leaves in a specific resolution tree of any $3$-tied link diagram. 

Furthermore, we obtain closed-form expressions for the AJ-bracket of $2$- and $3$- tied diagrams in terms of the Kauffman bracket of certain (classical) sublinks. These results shed light on the combinatorial structure of the AJ-bracket, providing explicit tools for its computation and paving the way for a potential categorification in the form of a tied Khovanov homology which strengths its classical counterpart \cite{Khovanov2000}. Our analysis may also offer new insights into the structure of other polynomial invariants extended to the tied setting, such as those studied in \cite{Aicardi2016, Chlouveraki2020, Aicardi2020, Aicardi2021}.

The paper is organized as follows. In Section~\ref{sec:preliminaries} we review the definition of tied links and the AJ-bracket, together with a preliminary analysis of resolution trees and AJ-states. Sections~\ref{section_bicolor} and \ref{section_tricolor} are devoted to the analysis of $2$- and $3$-tied links, respectively, where we present and prove our main results.

\bigskip

\noindent {\bf Acknowledgements:} The authors were partially supported by the project PID2024-157173NB-I00 funded by MCIN/AEI/10.13039/501100011033 and by FEDER, EU. The first author is supported by ANID, Beca Chile Doctorado en el Extranjero, Folio 72220167.

\section{Preliminaries}
\label{sec:preliminaries}
\subsection{Tied Links}

\begin{definition}
A tied link is a pair $(L, P)$, where $L$ denotes a classical link, and $P$ represents a partition of its components. Two tied links are said to be equivalent if the associated classical links are related by an ambient isotopy preserving the partition of their components. 
\end{definition}
  
We can think of a tied link as a colored link where components share the same color if and only if they belong to the same subset of the partition. Notice that the number of required colors in a tied link is bounded above by the number of components.  If the components of $L$ are partitioned into $n$ subsets (i.e., they are colored using $n$ colors), we say that $(L,P)$ is an $n$-tied link.

Throughout this paper we often drop $P$ from the notation, and think of the $n$-tied link $L$ as a classical link where each component has been assigned a label (color) in the set $\{1, 2, \ldots, n\}$. We endow an order in the set of colors given by the natural order in $\mathbb{N}$.
 
A tied link diagram $D$ is a labeled (colored) regular projection of a tied link. Similarly to the classical case, two tied link diagrams represent equivalent tied links if there exists a finite sequence of (classical) Reidemeister moves transforming one into the other while preserving the coloring of the components, up to color permutation. 

\subsection{Aicardi-Juyumaya Bracket and diagrams complexity}

In \cite{Aicardi2018} F. Aicardi and J. Juyumaya generalized the Kauffman bracket for classical links to the setting of tied links. In this section, we review their construction, which we call \textit{Aicardi-Juyumaya bracket} (AJ-bracket), recall some of its properties and introduce a complexity function for a tied link diagram, following~\cite{Cardenas2024}. We first introduce some notation. 

Let $D$ be a tied link diagram and consider two colors $i$ and $j$ associated to two partition subsets of its components, with $i < j$. We denote by $D_m, D_{d_1}, D_{d_2}, D_a$, and $D_b$ five tied link diagrams which are identical as classical diagrams everywhere but in a neighborhood of a crossing, as shown in \autoref{multidiagram}, where colors $i$ and $j$ are represented by black and red colors, respectively. As tied link diagrams, components colored by $j$ in $D_{d_1}$and $D_{d_2}$ turn out to be colored by $i$ in $D_m, D_a$ and $D_b$, with the rest of colors being preserved. Therefore if $D_{d_1}$ and $D_{d_2}$ represent $n$-tied links, then $D_m, D_a$ and $D_b$ represent $(n-1)$-tied links. 

\begin{figure}[H]
  \centering
  \includegraphics[height=2.5cm]{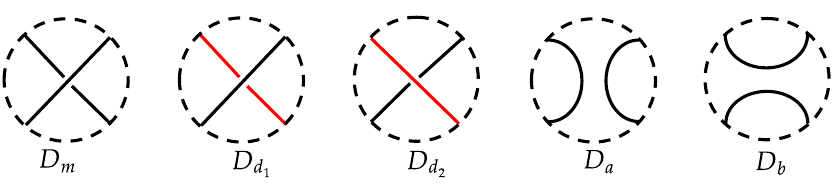}
  \caption{Local diagrams with colors $i$ (black) and $j$ (red) such that $i<j$.}
  \label{multidiagram}
\end{figure}

Given an $n$-tied link diagram $D$, denote by $D\sqcup\bigcirc$  (resp. $D\widetilde{\sqcup}\bigcirc$) the $(n+1)$-tied link diagram (resp. the $n$-tied link diagram) consisting of the disjoint union of $D$ together with the trivial diagram of the unknot colored by $n+1$ (resp. colored by one of the colors in $\{1, \ldots, n\}$).

\begin{theorem}\cite{Aicardi2018} \label{invariante}
There exists a unique function $\langle\langle\cdot\rangle\rangle \ : \{\mbox{tied links}\}\rightarrow\mathbb{Z}[A^{\pm1},c]$, defined by the following axioms:
\begin{enumerate}
    \item $\langle\langle \bigcirc  \rangle\rangle=1$,
    \item $\langle\langle D \sqcup \bigcirc\rangle\rangle=c\cdot\langle\langle D\rangle\rangle$,
    \item $\langle\langle D \widetilde{\sqcup} \bigcirc\rangle\rangle=-(A^2+A^{-2})\langle\langle D\rangle\rangle$,
    \item $\langle\langle \cdot \rangle\rangle $ is invariant under Reidemeister moves II y III.
    \item $\langle\langle D_{m} \rangle\rangle =
A\ \langle\langle D_a\rangle\rangle +
 {A^{-1}} \  \langle\langle D_b \rangle\rangle $,
\item $\langle\langle D_{d_1} \rangle\rangle  +  \langle\langle D_{d_2} \rangle\rangle =
\delta\left( \langle\langle D_a \rangle\rangle +
  \langle\langle D_b \rangle\rangle \right) $, with $\delta = A+A^{-1}$.
\end{enumerate}
\end{theorem}

Given a diagram $D_m$ with a distinguished crossing $x$ as drawn in \autoref{multidiagram}, we say that the diagram $D_a$ (resp. $D_b$) is obtained from $D_m$ by a smoothing of type $\bar{0}$ (resp. type $\bar{1}$) of the crossing $x$. Analogously, we say that the diagram $D_a$ (resp. $D_b$) is obtained from the tied link diagram $D_{d_1}$ by a smoothing of type $0$ (resp. type $1$) of the distinguished crossing, and that $D_{d_2}$ is obtained from $D_{d_1}$ by a smoothing of type $2$ of the distinguished crossing. Observe that the order of the colors is crucial when identifying diagrams $D_{d_1}$ and $D_{d_2}$. 

\begin{remark}\label{low-complexity}
When applying a smoothing of type $ \bar{0} $, $ \bar{1} $, or $ 2 $, the colors associated to the link components are preserved. However, in smoothings of types $0$ and $1$, axiom $(6)$ implies that the resulting arcs in $D_a$ and $D_b$ inherit the color of one of the components involved in the distinguished crossing (i.e., two subsets of the partition are merged into one). To avoid indeterminacy, if the colors of the involved components in $D_{d_1}$ are $i$ and $j$, where $i<j$, we declare that the resulting arcs in $D_a$ and $D_b$ (and therefore, all components colored by $i$ and $j$) inherit color $i$. 
\end{remark}

\autoref{smoothing} summarizes the previous discussion and encodes axioms $(5)$ and $(6)$ in Theorem~\ref{invariante}. 

\begin{figure}[H]
    \centering
    \includegraphics[scale=0.6]{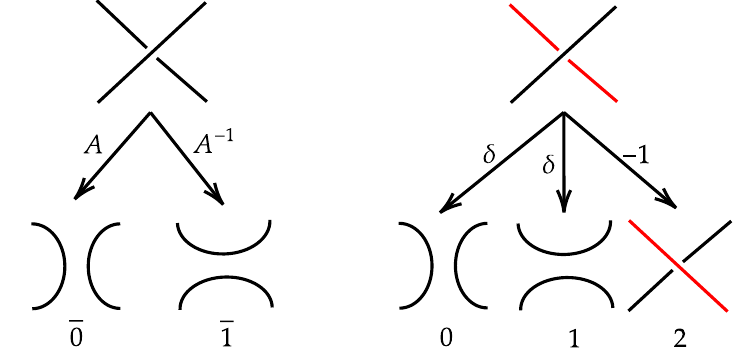}
    \caption{Resolution trees encoding axioms $(5)$ (left) and $(6)$ (right), where colors $i$ (black) and $j$ (red) satisfy $i<j$.}
    \label{smoothing}
\end{figure}

Consider a crossing $x$ in the tied link diagram $D$, and write $a$ (resp. $b$) for its upper (resp. lower) arc. We say that $x$ is an \textit{illegal crossing} if it satisfies one of the following conditions: 
\begin{enumerate}
	\item Both arcs $a$ and $b$ share the same color (as in the special crossing shown in $D_m$ in \autoref{multidiagram}). 
	\item The color $i$ associated to arc $a$ and the color $j$ associated to arc $b$ satisfy $i<j$ (as in the special crossing shown in $D_{d_1}$ in \autoref{multidiagram}). 
\end{enumerate}  
Classical crossings are those satisfying the description in (1) and we call them monochromatic, while crossings satisfying (2) are said to be dichromatic illegal crossings. Observe that applying axiom (5) to a monochromatic illegal crossing, or axiom (6) to a dichromatic illegal crossing, results into the smoothings shown in \autoref{smoothing}.

\begin{definition}
The complexity of a tied link diagram $D$ is defined as the pair $(C_T, C_I)$, where $C_T$ corresponds to the total number of crossings in $D$, and $C_I$ to its number of illegal crossings. 
\end{definition}

Diagram $D$ shown in \autoref{resolution-tree} contains two monochromatic (black) crossings, and two dichromatic illegal crossings, whose upper arcs are colored by $1$ (black) whereas the lower arcs are colored by $2$ (blue). Consequently, the complexity of this diagram is $(6,4)$.
	
Since $(C_T, C_I)\in\mathbb{N}^2$, we can equip the set of complexities of tied link diagrams with the lexicographic order.

\subsection{Resolution trees and AJ-states}

Given a tied link diagram $D$, we can construct a resolution tree rooted at $D$ whose vertices are labeled by diagrams of tied links obtained through the iterative application of the smoothings shown in \autoref{smoothing} (i.e., applying axioms (5) and (6) in \autoref{invariante} to a chosen illegal crossing in each step). That is, the children of each vertex $v$ correspond to the two (resp. three) diagrams obtained by performing type $ \bar{0} $ and $ \bar{1} $ smoothings (resp. type 0, 1, and 2 smoothings) on a monochromatic (resp. dichromatic illegal) crossing of the diagram depicted in $v$. At each edge of the tree, we keep the labels shown in \autoref{smoothing}. We set the process to finish when every leaf of the tree consists of a tied link diagram with no illegal crossings, that we call \textit{AJ-state} (see \autoref{AJ-state}). See \autoref{resolution-tree}, where diagrams $D_5$ to $D_8$ are AJ-states. 

\begin{remark}\label{low-complexity2}
The process of constructing a resolution tree always terminates, as the children of a vertex have smaller complexity than their parent. More precisely, smoothing an illegal crossing $x$ of a diagram $D$ with complexity $(m,k)$ yields the following outcomes (see \autoref{low-complexity}):
\begin{itemize}
\item If $x$ is monochromatic, then both resulting diagrams have complexity $(m-1,k-1)$.
\item If $x$ is dichromatic, then two resulting diagrams have complexity $(m-1,k')$ for some $k' \in \mathbb{N}$, and the third diagram has complexity $(m, k-1)$.
\end{itemize}
\end{remark}

Notice also that distinct vertices of a resolution tree could have the same associated diagram, so we will try to make a clear distinction between a vertex of the tree and the diagram which appears in that vertex. Nevertheless, to simplify the writing, we will talk about some feature of a vertex (for instance, its AJ-bracket), meaning the feature of the diagram associated to that vertex.

Given a resolution tree, the AJ-bracket of a vertex is equal to the sum of the AJ-brackets of its children, each one multiplied by the label of its connecting edge. Therefore, the AJ-bracket of the root is equal to the sum of the AJ-brackets of all leaves, each one multiplied by the product of the labels in the unique path connecting that leaf to the root. The AJ-bracket of each leaf can be computed using axioms (1), (2), (3) and (4) in \autoref{invariante}, as we will shortly see.

If $D$ contains more than one illegal crossing, the resolution tree described above (and therefore, the set of tied diagrams appearing at its leaves) is not unique, since the process depends on the chosen order of the crossings to be smoothed. Moreover, the set of illegal crossings might change when applying smoothings of type $0$ and $1$, as these smoothings do not preserve the colors of the components. 

\begin{definition}\label{AJ-state}
The Aicardi-Juyumaya states (AJ-states) associated to a tied link diagram $D$ are those diagrams containing no illegal crossings (i.e., with complexity $(n , 0)$) that can be obtained from $D$ by a finite sequence of smoothings of illegal crossings following axioms (5) and (6) in \autoref{invariante}.  In other words, they are the states that appear at a leaf of some resolution tree of $D$.
\end{definition}

Notice that, in an AJ-state, the components corresponding to the same color must be a family of disjoint circles (as there cannot be monochromatic crossings). Moreover, if circles of different colors $i<j$ overlap, the circle with color $i$ must be below the circle with color $j$ (as all dichromatic crossings must be legal). Therefore, after applying Reidemeister moves of type II and III (which do not modify the AJ-bracket by axiom (4) in \autoref{invariante}), we can assume that an AJ-state is a family of disjoint, not overlapping circles of distinct colors. If an AJ-state $D$ has $k$ colors and $s$ circles, its AJ-bracket is precisely 
$$
\langle\langle D\rangle\rangle = c^{k-1}(-A^2-A^{-2})^{s-k},
$$ 
by axioms (1), (2) and (3) in \autoref{invariante}.

If the resolution tree of a tied link diagram $D$ has a leaf with a single color, then the set of AJ-states of $D$ contains the set of Kauffman states associated to the (classical) diagram obtained by forgetting the colors of $D$.
As an example, in \autoref{resolution-tree} we can see an (incomplete) resolution tree. The diagrams $D_1$ to $D_4$ still have monochromatic crossings, which will be smoothed in the classical way. Each possible Kauffman state of $D$ will appear as a leaf of some $D_i$, for $i=1,\ldots,4$. The diagrams $D_5$ to $D_8$ are AJ-states in two colors, so they are already leaves of the resolution tree.

\begin{figure}[H]
	\centering
	\includegraphics[scale=0.55]{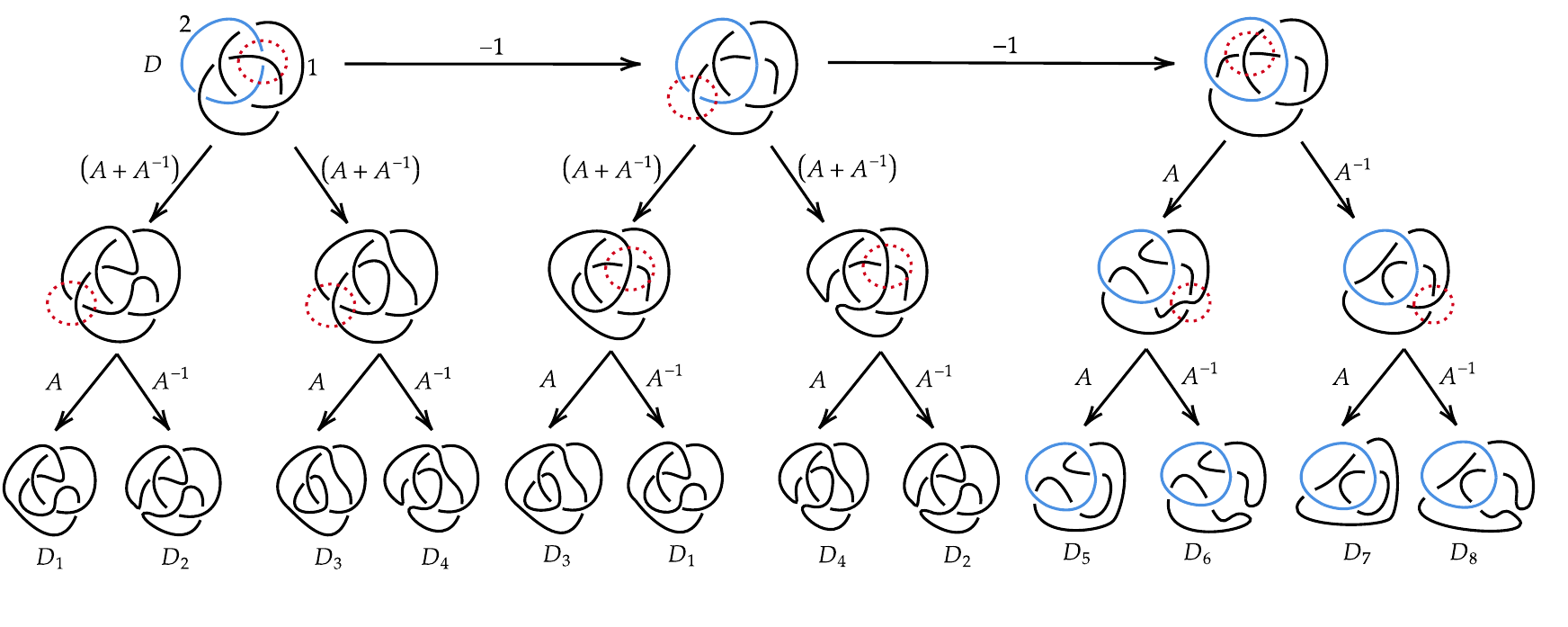}
	\caption{The first steps in the construction of a resolution tree of $D$. Diagrams $D_5$ to $D_8$ are AJ-states, whereas $D_1$ to $D_4$ are not.} 
	\label{resolution-tree}
\end{figure}

When computing the Kauffman bracket of a (classical) link diagram, one might consider different resolution trees, depending on the sequence of crossings to be iteratively smoothed. However, the number of leaves of the resolution tree is determined by the number of crossings of the diagram; more precisely, if $D$ contains $m$ crossings, then the number of leaves of any resolution tree is $2^m$. Moreover, each Kauffman state appears exactly once as a leaf in every resolution tree rooted at $D$, and therefore the total number of Kauffman states of $D$ is $2^m$. 

In the tied case the situation turns out to be a bit more convoluted. As in the classical case, when applying smoothings in \autoref{smoothing} to a tied link diagram $D$, one might consider different sequences of illegal crossings, resulting into different resolution trees. This time, given two resolution trees $T_1$ and $T_2$ of $D$, the set of AJ-states appearing as leaves of $T_1$ does not necessarily coincide with that of $T_2$. Nevertheless, we will prove in \autoref{section_bicolor} that, at least in the case of $2$--tied links, these two sets have the same cardinal, and moreover, both trees will have the same number of leaves.

\begin{remark}\label{repetitions}
We point out that AJ-states may appear multiple times at different leaves of a resolution tree. As an example, see the incomplete resolution tree shown in \autoref{resolution-tree}, i.e., where the (tied link) diagrams in the lower row sharing the same label are equal, but they have been obtained by applying a different sequence of smoothings. 

In general, a given AJ-state may appear in the subtrees hanging from two of the branches created when smoothing a dichromatic illegal crossing. For example, if $D$ contains two dichromatic illegal crossings, say $c_1$ and $c_2$, involving the same pair of colors, then performing a type $0$ smoothing on $c_1$ followed by a type $\bar{0}$ (resp. $\bar{1}$) smoothing on $c_2$ yields the same diagram as performing a type $2$ smoothing on $c_1$, followed by a type $0$ (resp. $1$) smoothing on $c_2$, and then a type $\bar{1}$ smoothing on $c_1$.
\end{remark} 

In \cite{Cardenas2024} it was shown that the total contribution of an AJ-state of $D$ to the bracket $ \langle \langle D \rangle \rangle$ (adding the contributions of all leaves associated to that particular AJ-state) does not depend on the chosen resolution tree. In particular, if an AJ-state $D_0$ appears in some leaves of a tree $T_1$, but it does not appear in any leave of a tree $T_2$ of the same diagram, then the total contribution of all leaves of $T_1$ corresponding to $D_0$ must be 0.

\subsection{Notation}\label{notation}

Given an $n$-tied link diagram $D$, we partition its set $C$ of crossings into the following subsets:
\begin{itemize}
	\item $X_{i,j}$ consists of the set of dichromatic illegal crossings whose arcs are colored by $i$ and $j$; we write $x_{i,j} = \#(X_{i,j})$. 
	\item $Y_{i,j}$ consists of the set of dichromatic legal crossings whose arcs are colored by $i$ and $j$; we write $y_{i,j} = \#(Y_{i,j})$.
	\item $Z_i$ consists of the set of monochromatic (illegal) crossings whose arcs are colored by $i$; we write $z_i = \#(Z_i)$.
\end{itemize}

Since $X_{i,j} = X_{j,i}$ and $Y_{i,j} = Y_{j,i}$, we keep the labeling where subindices are used in ascending order. We also define the sets $$X = \displaystyle\bigcup_{i<j} X_{i,j}, \quad \quad Y = \displaystyle\bigcup_{i<j} Y_{i,j}  \quad \mbox{ and } \quad Z = \displaystyle\bigcup_i Z_i,$$ and denote their cardinalities by $x$, $y$ and $z$, respectively. The number of crossings of $D$ is $\#(C)=x+y+z$.

In the subsequent sections, when studying the diagrams appearing in a resolution tree of $D$, the sets and numbers we just defined will always correspond to the diagram $D$, that is, to the root of the tree, unless otherwise stated.

Now let $S\subset C$ be a subset of the crossings of $D$. As usual, we denote $\{0,1\}^S$ to be the set of maps from $S$ to $\{0,1\}$. Then, for every $\sigma\in \{0,1\}^S$ we define 
$$
     k_{\sigma}=\#(\sigma^{-1}(0))-\#(\sigma^{-1}(1)).
$$ 
Notice that, if $S=\varnothing$, there is a unique possible $\sigma$ (the empty map), and $k_{\sigma}=0$.

Suppose now that $\sigma\in \{0,1,\bar{0},\bar{1}\}^S$ for some $S\subset C$. In this case we define:
$$
   e_{\sigma}=\#(\sigma^{-1}(\bar{0}))-\#(\sigma^{-1}(\bar{1})).
$$

Finally, given $\sigma\in \{0,1\}^S$, we denote by $D_{\sigma}$ the diagram obtained by smoothing each crossing $c\in S$ according to $\sigma(c)$ (or to $\overline{\sigma(c)}$, if the crossing is monochromatic). If some crossing $c\in S$ is dichromatic, the two colors must be merged into one (the smaller one). The order in which the crossings of $S$ are smoothed is irrelevant, hence $D_{\sigma}$ only depends on the original diagram $D$ and on the map $\sigma$. Notice that, if $S=\varnothing$, then $D_\sigma=D$.


\section{State sum model for 2-tied links}\label{section_bicolor}

For classical links, Kauffman developed a combinatorial method to compute the Jones polynomial via the so-called Kauffman bracket, whose defining relations are given by (1), (3) and (5) in \autoref{invariante}. Recall that the Kauffman states associated to a diagram $D$ are obtained by smoothing each crossing of $D$ in both possible ways (i.e., applying a $\bar{0}$ or a $\bar{1}$ smoothing). Therefore, given a state $s$ consisting of $k$ circles obtained by performing $\bar{1}$--smoothings to $r$ of the $m$ crossings of $D$ (and $\bar{0}$--smoothings of the remaining $m-r$ crossings), it follows from the defining relations that the contribution of $s$ to the bracket $\langle D \rangle$ is $A^{m-2r} (-A^2-A^{-2})^{k-1}$.

Taking the sum over all Kauffman states associated to $D$ one obtains the state sum formula for the Kauffman bracket of a (classical) link diagram:
\begin{align}
	\label{suma-estados}
	\langle D \rangle=\sum_{s}A^{m-2r}(-A^2-A^{-2})^{k-1}.
\end{align}

However, when considering the tied case, no expression analogous to the state sum formula \eqref{suma-estados} is known for the AJ-bracket. In this section we analyze the AJ-states associated to $2$-tied links, and provide a state sum formula to compute the AJ-bracket of those diagrams representing them. 

Using the notations in \autoref{notation}, we have the following:

\begin{theorem}\label{bicolor}
Let $D$ be a 2-tied link diagram with $m=x+y+z$ crossings. Then, in every resolution tree of $D$: 
\begin{enumerate}
    \item The number of leaves is $2^z+x2^m$.
    \item The number of dichromatic AJ-states is $2^z$.
    \item The number of monochromatic AJ-states is $0$ if $x=0$ and $2^m$ if $x>0$.
\end{enumerate}
\end{theorem}

\begin{proof}
First notice that, if $D_0$ is a 1-tied link diagram with $m$ crossings, then any resolution tree of $D_0$ will be a classical resolution tree, with $2^m$ leaves, all of them monochromatic.

Now let $D$ be a 2-tied link diagram. We proceed by induction on the complexity of $D$. If $D$ has complexity $(m,0)$, then it is an AJ-state, in which all crossings (if any) are dichromatic legal crossings. The only possible tree in this case consists of a single vertex (the root), and the result holds as $x=z=0$.
    
Suppose now that $D$ is a $2$-tied link diagram of complexity $(m, k)$, with $k>0$, and assume that the statement holds for any $2$-tied diagram with lower complexity. Let $T_D$ be a resolution tree rooted at $D$, and focus on the first crossing $c$ smoothed in $T_D$ (i.e, the illegal crossing whose smoothing corresponds to the root of $T_D$ and its 2 or 3 children). There are two possibilities: \\

\noindent \textbf{Case 1:} If $c$ is a monochromatic crossing, then $T_D$ contains two subtrees $T_{D_a}$ and $T_{D_b}$ rooted at the 2-tied link diagrams $D_a$ and $D_b$, respectively (see \autoref{smoothing}, where $D$ takes the role of $D_m$). By \autoref{low-complexity2}, both diagrams $D_a$ and $D_b$ have complexity $(m-1, k-1)$. In fact, when smoothing $c$, the number of dichromatic illegal crossings $x_a$ in $D_a$ (resp. $x_b$ in $D_b$) is preserved, while its number of monochromatic crossings $z_a$ (resp. $z_b$) is reduced by one: $$x_a=x_b=x \quad \mbox{ and } \quad z_a = z_b = z-1.$$ Therefore, by induction hypothesis, the number of leaves in each of the subtrees is $2^{z-1} + x2^{m-1}$, so $T_D$ has $2(2^{z-1}+x2^{m-1})=2^z+x2^m$ leaves. 

The induction hypothesis also implies that the number of dichromatic AJ-states in each of $T_{D_a}$ and $T_{D_b}$ is $2^{z-1}$. Observe that these states are different, since in order to get $D_a$ from $D$, the crossing $c$ was smoothed following a $\bar{0}$-label, while to get $D_b$ it was smoothed following a $\bar{1}$-label. Hence, $T_D$ contains $2\cdot2^{z-1}=2^z$ dichromatic AJ-states.

To compute the number of monochromatic AJ-states, we distinguish two cases: 
\begin{enumerate}
    \item If $x_a = x_b = 0$, the induction hypothesis implies that the number of monochromatic AJ-states for $T_{D_a}$ and $T_{D_b}$ is $0$, and so is for $T_D$. 
    \item If $x_a=x_b > 0$, the induction hypothesis implies that the number of monochromatic AJ-states for each subtree $T_{D_a}$ and $T_{D_b}$ is $2^{m-1}$, all of them being different (since $c$ is smoothed differently in $D_a$ and in $D_b$), so $T_D$ contains $2\cdot2^{m-1}=2^m$ monochromatic AJ-states.
\end{enumerate}

\noindent \textbf{Case 2:} If $c$ is a dichromatic illegal crossing (this can happen only if $x>0$), then $T_D$ contains three subtrees $T_{D_a}$, $T_{D_b}$ and $T_{D_{d_2}}$, rooted at diagrams $D_a$, $D_b$ and $D_{d_2}$, respectively (see \autoref{smoothing}, where the role of $D_{d_1}$ is taken by $D$). By \autoref{low-complexity2}, $D_a$ and $D_b$ have complexities $(m-1, m-1)$, since all their crossings are monochromatic, while $D_{d_2}$ has complexity $(m, k-1)$.

Since $D_a$ and $D_b$ are monochromatic and each one has $m-1$ crossings, it follows that the number of leaves in $T_{D_a}$ is $2^{m-1}$, and in $T_{D_b}$ as well. All those leaves correspond to different states (since $c$ is smoothed differently in $D_a$ and in $D_b$), hence the total number of monochromatic states in both subtrees is $2^m$. They correspond to all possible monochromatic states that can be obtained from $D$. Therefore, (3) is shown.

For the dichromatic diagram $D_{d_2}$, we have the following parameters:
$$
 \quad x_{d_2} = x-1, \quad m_{d_2}  = m \quad \mbox{ and }  \quad z_{d_2}  = z.
$$

Hence, the induction hypothesis implies that the number of leaves in  $T_{D_{d_2}}$ is $2^z + (x-1) 2^{m}$. Adding the $2^m$ leaves of $D_a$ and $D_b$ proves statement (1). 

Finally, statement (2) follows by applying the induction hypothesis to $T_{D_{d_2}}$, since $D_a$ and $D_b$ are monochromatic, and therefore do not contribute with dichromatic AJ-states to $T_D$. 
\end{proof}

Given a $2$-tied link diagram $D$, we will consider maps $\sigma : X \to \{0, 1\}$, whose domain is the set of dichromatic illegal crossings $X$.  Recall that, in this case, $D_\sigma$ is the diagram obtained from $D$ by smoothing each crossing $c_i \in X$ according to $\sigma(c_i)$. Notice that $D_{\sigma}$ is monochromatic unless $X=\varnothing$, in which case $D_{\sigma}=D$.

\begin{definition}
Given a $2$-tied link diagram $D$, a pseudo-AJ-state of $D$ is a diagram $D_{\sigma}$ for some $\sigma\in \{0,1\}^{X}$. We denote the set of pseudo-AJ-states associated to $D$ as $ps(D)$. Notice that the number of pseudo-AJ-states is $2^{x}$.
\end{definition}

\begin{theorem}\label{teobicolor}
Let $D$ be a 2-tied link diagram with $m=x+y+z$ crossings. Then, its AJ-bracket $\langle\langle D  \rangle\rangle $ can be computed in terms of the classical Kauffman bracket $\langle\cdot  \rangle$ as follows:
\begin{equation*}
            \langle\langle  D \rangle\rangle  \ = \ (-1)^x\langle D_1\rangle\langle D_2\rangle \ c +\displaystyle\sum_{\sigma\in \{0,1\}^X}H_{k_\sigma}\langle D_{\sigma}\rangle,
\end{equation*}
where $D_1$ and $D_2$ are the subdiagrams of $D$ colored by $1$ and $2$ respectively, and $H_k=A^k+(-1)^{k+1}A^{-k}$ for every integer $k$.
\end{theorem}

\begin{proof}
Let $X=\lbrace c_1,\dots,c_x\rbrace$ be the set of dichromatic illegal crossings of $D$. We will consider a resolution tree $T_D$ of $D$ whose first levels are constructed as follows:

\begin{figure}[h!]
   \centering
    \includegraphics[scale=0.9]{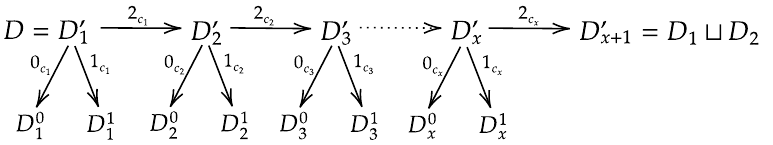}
    \caption{A sketch of the resolution tree of $D$ used in the proof of \autoref{teobicolor}. }
    \label{concretetree}
\end{figure}

First, we describe how to get the nodes of $T_D$  illustrated in \autoref{concretetree}: Starting from the root $D=D_1'$, we first smooth the crossings in $X$ in ascending order. More precisely, for $i\in \{1, \ldots, x\}$, in step $i$, we smooth the dichromatic illegal crossing $c_i$ of $D_i'$, giving rise to two monochromatic diagrams $D_i^0$ and $D_i^1$ obtained after applying a type $0$ and type $1$ smoothing to $c_i$, respectively, and a $2$-tied link diagram $D'_{i+1}$, obtained after applying a type $2$ smoothing. In the diagram $D_{x+1}'$, all dichromatic crossings are legal, hence $D_{x+1}'$ is equivalent via Reidemeister moves II and III to the $2$-tied link diagram $D_1\sqcup D_2$, and therefore $\langle\langle  D_{x+1}' \rangle\rangle  = \langle D_1\rangle\langle D_2\rangle c$. Hence, the contribution of $D_{x+1}'$ to the AJ-bracket of $D$ is
\begin{equation}\label{2-state}
    (-1)^{x} \langle D_1\rangle\langle D_2\rangle c.
\end{equation}

Notice that, if $X=\varnothing$, then $D=D_{x+1}'=D_1'$ and $\langle\langle  D \rangle\rangle = \langle D_1\rangle\langle D_2\rangle c$. In this case we have $k_\sigma = 0$ for the only possible map $\sigma$, hence $H_{k_\sigma}=A^{k_\sigma}+(-1)^{k_{\sigma}+1}A^{-k_{\sigma}}=1-1=0$, and the result holds in this case. So we can assume that $X\neq \varnothing$.

For each $j \in \{1, \ldots, x\}$, we define a subtree $T_j$ of $T_D$ rooted at $D_j'$ as follows: the children of the root are the immediate descendants $D_j^0$ and $D_j^1$ (i.e. we deliberately exclude from $T_j$ the third descendant $D_{j+1}'$). In the subsequent $x-1$ steps, we iteratively apply all possible smoothings of types $\overline{0}$ and $\overline{1}$ to the crossings in the set $X\setminus\{c_j\}$. The resulting subtree $T_j$ thus contains $2^x$ distinct pseudo-AJ-states. For our purposes, we assume that we interrupt the construction of the full resolution tree $T_D$ at this stage, i.e., we assume that $T_j$ contains $2^x$ leaves.

Observe that the set of pseudo-AJ-states of the subtree $T_j$ does not depend on $j$, since each of them is obtained from $D$ by forgetting the colors and smoothing the crossings in $X$ in all possible ways.  

Now, given a map $\sigma : X \to \{0, 1\}$, define the vertex set  $$V_\sigma=\lbrace V_{\sigma,j}, \quad   1\leq j\leq x \rbrace,$$ where $V_{\sigma,j}$ is the vertex of $T_j$ associated to diagram $D_{\sigma}$. The diagram $D_{\sigma}$ appearing in vertex $V_{\sigma,j}$ is obtained after applying the following sequence of transformations to diagram $D$:

\begin{enumerate}
    \item Apply smoothings of type 2 to the crossings $c_1,\dots,c_{j-1}$.
    \item Apply a smoothing of type $\sigma(c_j)$ to the crossing $c_j$. 
    \item For $k\in\{j+1,\dots,x\}$, apply a smoothing of type $\overline{\sigma(c_k)}$ to $c_k$.
    \item For $k\in\{1,\dots,j-1\}$, apply a smoothing of type $\overline{1-\sigma(c_k)}$ to $c_k$ (recall that in step (1) crossing $c_k$ was mirrored).
    \end{enumerate}
    
    The previous sequence of smoothings is summarized in the $x$-tuple  
$$
\sigma_{(j)}=(\overline{1-\sigma(c_1)},\dots,\overline{1-\sigma(c_{j-1})}, \, \sigma(c_j), \, \overline{\sigma(c_{j+1})},\dots,\overline{\sigma(c_x)}),
$$
which represents a function in $\{ 0, 1, \overline{0},\overline{1}\}^X$. Recall that in \autoref{notation} we defined
$e_{\sigma_{(j)}}= \#(\sigma_{(j)}^{-1}(\overline{0}))-\#(\sigma_{(j)}^{-1}(\overline{1}))$.

Now we analyze the contribution $P_{\sigma,j}$ of the descendants of the vertex $V_{\sigma,j}$ to the AJ-bracket of $D$. Since the associated diagram $D_\sigma$ is monochromatic, $\langle\langle D_{\sigma}\rangle\rangle = \langle D_{\sigma}\rangle$. Hence, the contribution is equal to $\langle D_{\sigma}\rangle$ multiplied by the product of all edge's labels in the unique path connecting $V_{\sigma,j}$ to the root of $T_D$. More precisely:

\begin{align*}
    P_{\sigma,j}&=(-1)^{j-1}(A+A^{-1})A^{e_{\sigma_{(j)}}} \,  \langle D_\sigma \rangle \\
    &=(-1)^{j-1}(A^{e_{\sigma_{(j)}}+1}+A^{e_{\sigma_{(j)}}-1}) \,  \langle D_\sigma \rangle \\
    &=(-1)^{j-1}(A^{a_{\sigma_{(j)}}}+A^{b_{\sigma_{(j)}}}) \,  \langle D_\sigma \rangle,
\end{align*}

\noindent where $a_{\sigma_{(j)}}$ and $b_{\sigma_{(j)}}$ are defined by:\\

$a_{\sigma_{(j)}} = \left\{ \begin{array}{cl}
e_{\sigma_{(j)}}+1 & \text{ if } \ \sigma(c_j)=0, \\
e_{\sigma_{(j)}}-1 & \text{ if } \ \sigma(c_j)=1;
\end{array} \right.$ \ and  \ 
$b_{\sigma_{(j)}} = \left\{ \begin{array}{cl}
e_{\sigma_{(j)}}-1 & \text{ if } \ \sigma(c_j)=0, \\
e_{\sigma_{(j)}}+1 & \text{ if } \ \sigma(c_j)=1. 
\end{array} \right.$\\ \\

\begin{claim}\label{claim1}
For every $j \in \{1,\dots,x-1\}$,  one has $b_{\sigma_{(j)}}=a_{\sigma_{(j+1)}}$.
\end{claim} 
\begin{proof}
The $x$-tuples $\sigma_{(j)}$ and $\sigma_{(j+1)}$ are equal in all coordinates but the ones in positions $j$ and $j+1$. The proof of the claim consists of a case-by-case analysis of the four possibilities for these coordinates in $\sigma_{(j)}$ and $\sigma_{(j+1)}$ (we omit writing the coordinates in other positions): 

\begin{enumerate}
	\item If $\sigma(c_j)=0=\sigma(c_{j+1})$,  then $\sigma_{(j)}=(\dots,0,\overline{0},\dots)$ and $\sigma_{(j+1)}=(\dots,\overline{1},0,\dots)$. Hence $e_{\sigma_{(j+1)}}-e_{\sigma_{(j)}}=-2$, and therefore $b_{\sigma_{(j)}}=e_{\sigma_{(j)}}-1=e_{\sigma_{(j+1)}}+1=a_{\sigma_{(j+1)}}$. \\
	
	\item If $\sigma(c_j)=1=\sigma(c_{j+1})$, then $\sigma_{(j)}=(\dots,1,\overline{1},\dots)$ and $\sigma_{(j+1)}=(\dots,\overline{0},1,\dots)$. Hence $e_{\sigma_{(j+1)}}-e_{\sigma_{(j)}}=2$, and therefore $b_{\sigma_{(j)}}=e_{\sigma_{(j)}}+1=e_{\sigma_{(j+1)}}-1=a_{\sigma_{(j+1)}}$. \\
	
	\item  If $\sigma(c_j)=0$ and $\sigma(c_{j+1})=1$, then $\sigma_{(j)}=(\dots,0,\overline{1},\dots)$ and $\sigma_{(j+1)}=(\dots,\overline{1},1,\dots)$. Hence $e_{\sigma_{(j)}}=e_{\sigma_{(j+1)}}$, and therefore $b_{\sigma_{(j)}}=e_{\sigma_{(j)}}-1=e_{\sigma_{(j+1)}}-1=a_{\sigma_{(j+1)}}$. \\
	
	\item  If $\sigma(c_j)=1$ and $\sigma(c_{j+1})=0$, then $\sigma_{(j)}=(\dots,1,\overline{0},\dots)$ and $\sigma_{(j+1)}=(\dots,\overline{0},0,\dots)$. Hence $e_{\sigma_{(j)}}=e_{\sigma_{(j+1)}}$, and therefore $b_{\sigma_{(j)}}=e_{\sigma_{(j)}}+1=e_{\sigma_{(j+1)}}+1=a_{\sigma_{(j+1)}}$.  
\end{enumerate}
\end{proof}

\begin{claim}\label{claim2}
For every map $\sigma: X \to \{0,1\}$, one has	$a_{\sigma_{(1)}}=k_\sigma=-b_{\sigma_{(x)}}$, where we recall that $k_\sigma=\#(\sigma^{-1}(0))-\#(\sigma^{-1}(1))$.
\end{claim}
\begin{proof}
For every $i \in \{1, \ldots, x\}$, write $\sigma(c_i) = s_i$. 

We first analyze $a_{\sigma_{(1)}}$. To do this, observe that  $\sigma_{(1)}=(s_1,\overline{s_2},\dots,\overline{s_{x-1}},\overline{s_{x}})$, and therefore:

\begin{enumerate}
	\item 	If $s_1=0$, then 
$$
   a_{\sigma_{(1)}}=e_{\sigma_{(1)}}+1=\#(\sigma_{(1)}^{-1}(\overline{0}))- \#(\sigma_{(1)}^{-1}(\overline{1}))+1=\#(\sigma^{-1}(0))-\#(\sigma^{-1}(1))=k_\sigma.
$$ 
	\item  If $s_1=1$, then 
$$
    a_{\sigma_{(1)}}=e_{\sigma_{(1)}}-1=\#(\sigma_{(1)}^{-1}(\overline{0}))-
    \#(\sigma_{(1)}^{-1}(\overline{1}))-1=\#(\sigma^{-1}(0))-\#(\sigma^{-1}(1))=k_\sigma.
$$
	
\end{enumerate} 
	
To analyze $b_x$, observe that  $\sigma_{(x)}=(\overline{1-s_1},\overline{1-s_2},\dots,\overline{1-s_{x-1}},\, s_{x})$, and therefore:

\begin{enumerate}
	\item If $s_x=0$, then
$$
   b_{\sigma_{(x)}}=e_{\sigma_{(x)}}-1=\#(\sigma_{(x)}^{-1}(\overline{0}))-
   \#(\sigma_{(x)}^{-1}(\overline{1}))-1=\#(\sigma^{-1}(1))-\#(\sigma^{-1}(0))=-k_\sigma.
$$ 
	
\item If $s_x=1$, then 
$$
   b_{\sigma_{(x)}}=e_{\sigma_{(x)}}+1=\#(\sigma_{(x)}^{-1}(\overline{0}))-
   \#(\sigma_{(x)}^{-1}(\overline{1}))+1=\#(\sigma^{-1}(1))-\#(\sigma^{-1}(0))=-k_\sigma.
$$
\end{enumerate}
\end{proof}

We now collect in the polynomial $P_\sigma$ the contributions of all descendants of the vertices $V_{\sigma,j}$ for $j=1,\ldots,x$, which are precisely the vertices of $T_D$ whose associated diagram is $D_{\sigma}$:
\begin{align*}
    P_{\sigma}&=\sum_{j=1}^xP_{\sigma,j}   \\
    &=\sum_{j=1}^x(-1)^{j-1}(A^{a_{\sigma_{(j)}}}+A^{b_{\sigma_{(j)}}}) \,  \langle D_\sigma \rangle \\
    &=\left(A^{a_{\sigma_{(1)}}}+\sum_{j=1}^{x-1}(-1)^{j-1}(A^{b_{\sigma_{(j)}}}-A^{a_{\sigma_{(j+1)}}})+(-1)^{x-1}A^{b_{\sigma_{(x)}}} \right) \,  \langle D_\sigma \rangle\\
    &= \left( A^{a_{\sigma_{(1)}}}+(-1)^{x-1}A^{b_{\sigma_{(x)}}} \right) \,  \langle D_\sigma \rangle,
\end{align*}

\noindent where we used Claim~\ref{claim1} in the last equality.

Finally, using Claim~\ref{claim2} and noting that $x$ and $k_\sigma$ have the same parity, we obtain:
\begin{equation*}
    P_{\sigma}= \left(A^{k_\sigma}+(-1)^{k_\sigma+1}A^{-k_\sigma} \right) \, \langle D_\sigma \rangle = H_{k_\sigma}\langle D_\sigma \rangle .
\end{equation*}

Taking the sum over all possible $\sigma\in \{0,1\}^{X}$ and including the expression \eqref{2-state}  completes the proof of the theorem. 
\end{proof}

\section{State sum model for 3-tied links}\label{section_tricolor}

\subsection{Trichromatic case} 

Let $D$ be a 3-tied link diagram with $m=x_{1,2}+x_{1,3}+x_{2,3}+y_{1,2}+y_{1,3}+y_{2,3}+z_1+z_2+z_3$ crossings (see \autoref{notation} for notation). We order the sets of crossings as follows $$X_{1,2}<X_{1,3}<X_{2,3}<Y_{1,2}<Y_{1,3}<Y_{2,3}<Z,$$

\noindent and we order the crossings in such a way that if $c_i\in A$ and $c_j\in B$ with $A<B$, then $i<j$.

Given the previous enumeration of the crossings of $D$, we construct a particular resolution tree ($T_D$) imposing that, at each vertex, the dichromatic illegal crossing with the smallest index is smoothed, and if no illegal dichromatic crossing remains, then the monochromatic crossing with the smallest index is smoothed.

It is important to remark that the sets $X_{i,j}$, $Y_{i,j}$ and $Z$, and the order of the crossings, are established once and for all at the root $D$ of the tree, but that the set of dichromatic illegal, dichromatic legal, or monochromatic crossings may vary at distinct vertices of $T_D$. For example, a dichromatic illegal crossing $c$ in a vertex $v$ of $T_D$ may belong to $Y$ (if $c$ was dichromatic legal in $D$), or a monochromatic crossing in $v$ may belong to $X$ (if it was dichromatic illegal in $D$). 

We know that in every resolution tree of a 3-tied link diagram, for every leaf, there are at most two smoothings of type 0 or 1 in the path connecting it to the root, because these smoothings fuse two colors into one. Then, the leaves (corresponding to not neccessarily different AJ-states) of $T_D$ can be classified into 7 families, that we describe by indicating the pair of colors (of $D$) involved in each of the aforementioned smoothings. 

\begin{itemize}
    \item $\Gamma_1:$ $\varnothing$
    \item $\Gamma_2:$ $\{(1,2)\}$
    \item $\Gamma_3:$ $\{(1,3)\}$
    \item $\Gamma_4:$ $\{(2,3)\}$
    \item $\Gamma_5:$ $\{(1,2),(1,3)\}$
    \item $\Gamma_6:$ $\{(1,2),(2,3)\}$
    \item $\Gamma_7:$ $\{(1,3),(2,3)\}$
\end{itemize}

For example, $\Gamma_6$ is the set of leaves obtained from the root by applying a sequence of smoothings including a smoothing of type 0 or 1 to a crossing from $X_{1,2}$ and another one to a crossing in $X_{2,3}$. On the other hand, in the path from $D$ to a leaf in $\Gamma_7$, some smoothing of type $0$ or $1$ has been applied to a crossing from $X_{1,3}$; after such a smoothing, components originally colored by color $3$ becomes colored by $1$, hence the crossings from $X_{2,3}$ become legal and the crossings from $Y_{2,3}$ become illegal; then there is a smoothing of type $0$ or $1$ applied to an element of $Y_{2,3}$. 

In \autoref{coloredtrees}, both the resolution tree described above and the sets of leaves $\Gamma_i$ are schematized. The notation $\{i,j\}$ over a horizontal edge indicates that all illegal crossings in $X_{i,j}\cup Y_{i,j}$ have been smoothed by a smoothing of type $2$, while the notation $(i,j)$ over a vertical edge indicates that a smoothing of type $0$ or $1$ has been performed in a dichromatic illegal crossing from $X_{i,j}\cup Y_{i,j}$. Note that to obtain the leaves in each family $\Gamma_i$ all monochromatic crossings must be smoothed. 

\begin{figure}[h!]
   \centering
    \includegraphics[scale=0.6]{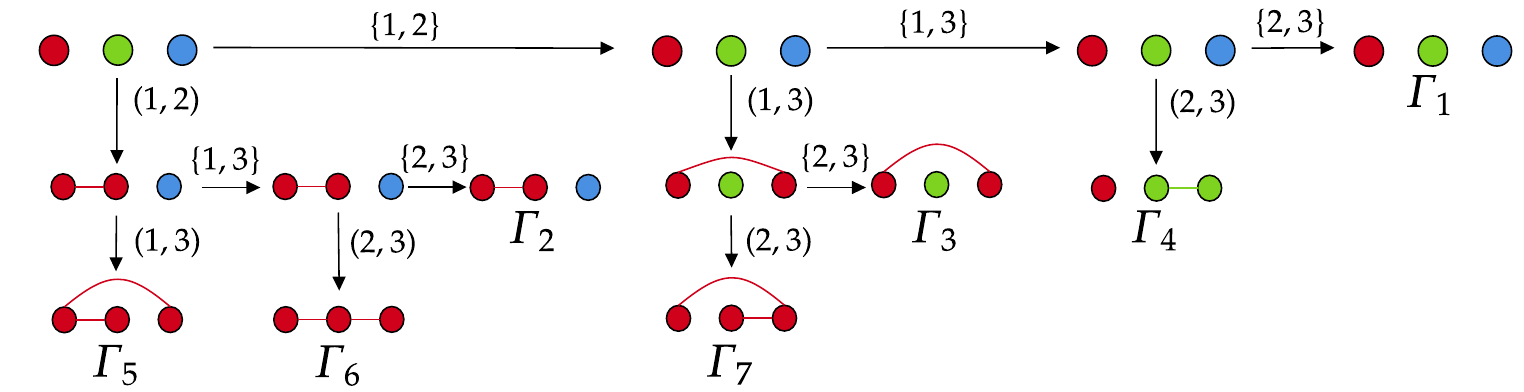}
    \caption{}
    \label{coloredtrees}
    \end{figure}

With the previous notation, we note that the leaves in $\Gamma_1$ are trichromatic, the leaves in $\Gamma_2$, $\Gamma_3$ and $\Gamma_4$ are dichromatic, and the leaves in $\Gamma_5$, $\Gamma_6$ and $\Gamma_7$ are monochromatic.

\begin{prop}\label{tricolor}
Let $D$ be a 3-tied link diagram with $m=x_{1,2}+x_{1,3}+x_{2,3}+y_{1,2}+y_{1,3}+y_{2,3}+z_1+z_2+z_3$ crossings. Then: 

\begin{enumerate}
    \item $\Gamma_1$ contains $2^z$ leaves and $2^z$ AJ-states.
    \item $\Gamma_2$ contains $x_{1,2}\cdot2^{x_{1,2}+y_{1,2}+z}$ leaves and, if $\Gamma_2\neq \varnothing$, $2^{x_{1,2}+y_{1,2}+z}$ AJ-states.
    \item $\Gamma_3$ contains $x_{1,3}\cdot2^{x_{1,3}+y_{1,3}+z}$ leaves and, if $\Gamma_3\neq \varnothing$, $2^{x_{1,3}+y_{1,3}+z}$ AJ-states.
    \item $\Gamma_4$ contains $x_{2,3}\cdot2^{x_{2,3}+y_{2,3}+z}$ leaves and, if $\Gamma_4\neq \varnothing$, $2^{x_{2,3}+y_{2,3}+z}$ AJ-states.
    \item $\Gamma_5$ contains $x_{1,2}\cdot x_{1,3}\cdot2^m$ leaves and, if $\Gamma_5\neq \varnothing$, $2^m$ AJ-states.
    \item $\Gamma_6$ contains $x_{1,2}\cdot x_{2,3}\cdot2^m$ leaves and, if $\Gamma_6\neq \varnothing$, $2^m$ AJ-states.
    \item $\Gamma_7$ contains $x_{1,3}\cdot y_{2,3}\cdot2^m$ leaves and, if $\Gamma_7\neq \varnothing$, $2^m$ AJ-states.
\end{enumerate}
\end{prop}

\begin{proof}
First, observe that in order to obtain the leaves in $\Gamma_1$, we must perform, in the prescribed order, smoothings of type 2 at each illegal dichromatic crossing in $X_{1,2}\cup X_{1,3}\cup X_{2,3}$, followed by smoothings of type $\bar{0}$ or $\bar{1}$ at each crossing in $Z$. The different choices of $\bar{0}$ or $\bar{1}$ for each crossing in $Z$ produce all the vertices of $\Gamma_1$, which are all distinct. Consequently, we obtain $2^z$ leaves, corresponding to the same number of AJ-states.

To enumerate the possible paths in $T_D$ going from $D$ to a vertex in $\Gamma_2$, we have exactly the following choices: which crossing $c_i\in X_{1,2}$ is smoothed by a smoothing of type 0 or 1 (there are $x_{1,2}$ choices), the type of smoothing applied to $c_i$ (either 0 or 1), and the type of smoothing (either $\bar{0}$ or $\bar{1}$) applied to the monochromatic crossings, which are precisely the crossings in $(X_{1,2}\setminus\{c_i\})\cup Y_{1,2}\cup Z$ (notice that, after the smoothing of $c_i$, all other crossings in $X_{1,2}$ and $Y_{1,2}$ become monochromatic). There is no other possible choice, as the original order of the crossings determines the order in which the smoothings are applied, and all crossings in $X_{1,3}$ and $X_{2,3}$ must be smoothed using a smoothing of type 2. It follows that there are exactly $x_{1,2}\cdot 2^{x_{1,2}+y_{1,2}+z}$ vertices in $\Gamma_2$. If $\Gamma_2\neq \varnothing$, that is, if $x_{1,2}\neq 0$, we notice that the final state is determined by the type of smoothing applied to each crossing of $X_{1,2}\cup Y_{1,2}\cup Z$, but not on the particular choice of the crossing $c_i$. Therefore, if $\Gamma_2\neq \varnothing$ the number of states is $2^{x_{1,2}+y_{1,2}+z}$.

The arguments for $\Gamma_3$ and $\Gamma_4$ are analogous to that for $\Gamma_2$.

Now we can enumerate the paths in $T_D$ going from $D$ to a leaf in $\Gamma_5$. The choices that determine each path are the following: a crossing $c_i\in X_{1,2}$ to which apply a smoothing of type 0 or 1 ($x_{1,2}$ choices); the type of smoothing applied to $c_i$ (2 choices); a crossing $c_j\in X_{1,3}$ to which apply a smoothing of type 0 or 1 ($x_{1,3}$ choices); the type of smoothing applied to $c_j$ (2 choices); the type of smoothing ($\bar{0}$ or $\bar{1}$) applied to each monochromatic crossing. The last step is applied to all $m$ crossings except $c_i$ and $c_j$, since they all become monochromatic after the smoothings of $c_i$ and $c_j$. Therefore, the total number of leaves in $\Gamma_5$ is $x_{1,2}\cdot x_{1,3}\cdot 2^{m}$. If $\Gamma_5\neq \varnothing$, that is, if $x_{1,2}\cdot x_{1,3}\neq 0$, the resulting AJ-state at the end of each path depends only on the choices of smoothings (0 or 1 for $c_i$ and $c_j$, $\bar{0}$ or $\bar{1}$ for the remaining crossings), and not on the choices of $c_i$ and $c_j$. Therefore, if $\Gamma_5\neq \varnothing$ the number of states is $2^m$. 

For $\Gamma_6$ and $\Gamma_7$, the procedure is analogous to that for $\Gamma_5$. However, in the case of $\Gamma_7$, note that when a smoothing of type 0 or 1 is performed at a crossing in $X_{1,3}$, the crossings in $X_{2,3}$ become legal and the crossings in $Y_{2,3}$ become illegal. Therefore, the second smoothing of type 0 or 1 must be applied to a crossing in $Y_{2,3}$.
\end{proof}

\begin{cor}
Let $D$ be a 3-tied link diagram with $m=x_{1,2}+x_{1,3}+x_{2,3}+y_{1,2}+y_{1,3}+y_{2,3}+z_1+z_2+z_3$ crossings. Then, for $T_D$: 
\begin{enumerate}
        \item The number of trichromatic AJ-states is $2^z$.
        \item The number of dichromatic AJ-states is 
        $$
            (\alpha_{1,2}2^{x_{1,2}+y_{1,2}}+\alpha_{1,3}2^{x_{1,3}+y_{1,3}}+\alpha_{2,3}2^{x_{2,3}+y_{2,3}})2^z,
        $$ 
        where $\alpha_{i,j}=1$ if $x_{i,j}>0$ and $\alpha_{i,j}=0$ otherwise. 
        \item The number of monochromatic AJ-states is either $2^m$ or 0.
\end{enumerate}
\end{cor}

\begin{proof}
All three items follow from \autoref{tricolor}.
(1) is straightforward.
(2) is obtained by adding the number of AJ-states in $\Gamma_2$, $\Gamma_3$, and $\Gamma_4$, since each set contains AJ-states with components in different partition blocks. The number $\alpha_{i,j}$ determines whether the corresponding set $\Gamma_k$ is empty or not.
(3) follows from the fact that if $\Gamma_i$ and $\Gamma_j$, with $i,j\in\{5,6,7\}$ are both nonempty, their corresponding AJ-states are the same, as in both cases all crossings of $D$ are smoothed. Hence, if some $\Gamma_i$ with $i\in\{5,6,7\}$ is nonempty, the number of monochromatic AJ-states is $2^m$, and if $\Gamma_5=\Gamma_6=\Gamma_7=\varnothing$, the number of monochromatic AJ-states is 0.
\end{proof}

If $D$ is a 3-tied link diagram with $D_1$, $D_2$ and $D_3$ the subdiagrams associated with colors 1, 2, and 3 respectively, we denote by $D\setminus D_i$ the subdiagram obtained by erasing all the components in $D_i$ from $D$. Additionally, we set $\langle\langle D(\Gamma_i)\rangle\rangle:=\displaystyle\sum_{s\in\Gamma_i}\langle\langle s\rangle\rangle$, where the AJ-bracket of a vertex means the AJ-bracket of its corresponding diagram. That is, $\langle\langle D(\Gamma_i)\rangle\rangle$ is the contribution of all leaves in $\Gamma_i$ to the polynomial $\langle\langle D\rangle\rangle$.

We know that $\Gamma_1$ corresponds to the trichromatic leaves. $\Gamma_2$, $\Gamma_3$ and $\Gamma_4$ correspond to the dichromatic leaves, but in each case the two colors which merge are distinct, so we will treat these three cases separately. The sets $\Gamma_5$, $\Gamma_6$ and $\Gamma_7$ correspond to the monochromatic leaves, and the leaves in each case are exactly the same (provided that the corresponding set of leaves is nonempty). Hence, in the following result we will add up the contribution of these three sets of leaves, denoting the set of monochromatic leaves as $M=\Gamma_5\cup \Gamma_6\cup \Gamma_7$, and setting $\langle\langle D(M)\rangle\rangle:=\displaystyle\sum_{s\in M}\langle\langle s\rangle\rangle$.

Finally, if $S,T\subset C$ are two disjoint subsets of crossings of $D$, given $\sigma\in \{0,1\}^{S}$ and $\tau\in \{0,1\}^{T}$, we define $\sigma\cup \tau\in \{0,1\}^{S\cup T}$ as the map sending $c$ to $\sigma(c)$ if $c\in S$ and to $\tau(c)$ if $c\in T$. 

We can now give an explicit formula for the contributions of each of the above sets of leaves to $\langle\langle D\rangle\rangle$.

\begin{theorem}
Let $D$ be a 3-tied link diagram with $D_1$, $D_2$ and $D_3$ the subdiagrams associated with colors 1, 2, and 3 respectively, and $m=x_{1,2}+x_{1,3}+x_{2,3}+y_{1,2}+y_{1,3}+y_{2,3}+z_1+z_2+z_3$ crossings. Then:

    \begin{enumerate}
        \item $\langle\langle D(\Gamma_1)\rangle\rangle=(-1)^{x}c^2 \langle D_1\rangle\langle D_2\rangle\langle D_3\rangle$ \medskip
        
        \item $\displaystyle\langle\langle D(\Gamma_2)\rangle\rangle = (-1)^{x_{1,3}+ x_{2,3}}c\,\langle D_3\rangle \sum_{\sigma\in \{0,1\}^{X_{1,2}}} H_{k_\sigma}\langle (D\setminus D_3)_{\sigma}\rangle$
            
        \item $\displaystyle\langle\langle D(\Gamma_3)\rangle\rangle = (-1)^{x_{1,2}+ x_{2,3}}c\,\langle D_2\rangle \sum_{\sigma\in \{0,1\}^{X_{1,3}}} H_{k_\sigma}\langle (D\setminus D_2)_{\sigma}\rangle$
            
        \item $\displaystyle\langle\langle D(\Gamma_4)\rangle\rangle = (-1)^{x_{1,2}+ x_{1,3}}c\,\langle D_1\rangle \sum_{\sigma\in \{0,1\}^{X_{2,3}}} H_{k_\sigma}\langle (D\setminus D_1)_{\sigma}\rangle$

      \item $\displaystyle  \langle\langle D(M) \rangle\rangle = \sum_{\sigma\in \{0,1\}^{X_{1,2}}}\sum_{\tau\in \{0,1\}^{X_{1,3}}}\sum_{\mu\in \{0,1\}^{X_{2,3}}}\sum_{\phi\in \{0,1\}^{Y_{2,3}}}{S_{k_\sigma,k_\tau,k_\mu,k_\phi}
          \langle D_{\sigma\cup\tau\cup\mu\cup\phi}\rangle}$
            
    \end{enumerate}
\noindent where $H_{k}=A^{k}+(-1)^{k+1}A^{-k}$ for every integer $k$, and  
$$
S_{k_\sigma,k_\tau,k_\mu,k_\phi}=H_{k_\sigma+k_\phi}H_{k_\tau+k_\mu} +(-1)^{k_\sigma+k_\tau-1}A^{-k_\sigma-k_\tau}H_{k_\mu}H_{k_\phi}.
$$
\end{theorem}

\begin{proof}

Let us first study $\Gamma_1$ and $\Gamma_4$. Every path in $T_D$ from the root $D$ to a leaf in $\Gamma_1\cup \Gamma_4$ starts by applying smoothings of type 2 to every crossing in $X_{1,2}\cup X_{1,3}$ (see \autoref{coloredtrees}). After applying all these smoothings, we obtain a vertex $v$ of $T_D$ whose diagram is equivalent, via Reidemeister moves II and III, to $D_1\sqcup(D\setminus D_1)$, where the color of $D_1$ is different from the two colors in $D\setminus D_1$. Since $D_1$ is monochromatic, we have $\langle\langle D_1\rangle\rangle=\langle D_1\rangle$. Notice that the leaves in $\Gamma_1\cup \Gamma_4$ are precisely those which are descendants of $v$. Therefore 
\begin{align*}
\langle\langle D(\Gamma_1) \rangle\rangle + \langle\langle D(\Gamma_4) \rangle\rangle & =  (-1)^{x_{1,2}+x_{1,3}}\langle\langle D_1\sqcup(D\setminus D_1) \rangle\rangle \\ & =  (-1)^{x_{1,2}+x_{1,3}} c\, \langle D_1\rangle \, \langle\langle D\setminus D_1\rangle\rangle.
\end{align*}
We can then apply \autoref{teobicolor} to the 2-tied link diagram $D\setminus D_1$, and hence obtain:

$$
    \langle\langle D(\Gamma_1)\rangle\rangle + \langle\langle D(\Gamma_4)\rangle\rangle = (-1)^{x_{1,2}+ x_{1,3}}c\,\langle D_1\rangle\langle\langle D\setminus D_1\rangle\rangle
$$    
$$
     = (-1)^{x_{1,2}+ x_{1,3}}c\,\langle D_1\rangle \left[(-1)^{x_{2,3}}\langle D_2\rangle\langle D_3\rangle c+\sum_{\sigma\in \{0,1\}^{X_{2,3}}} H_{k_\sigma}\langle (D\setminus D_1)_{\sigma}\rangle\right]
$$
$$
     = (-1)^{x}c^2 \langle D_1\rangle\langle D_2\rangle\langle D_3\rangle + (-1)^{x_{1,2}+ x_{1,3}}c\,\langle D_1\rangle \sum_{\sigma\in \{0,1\}^{X_{2,3}}} H_{k_{\sigma}}\langle (D\setminus D_1)_{\sigma}\rangle. 
$$

Note that the first term in the final sum is $\langle\langle D(\Gamma_1)\rangle\rangle$, and the second is $\langle\langle D(\Gamma_4)\rangle\rangle$.

For the cases $\Gamma_2$ and $\Gamma_3$ the argument is analogous. One simply considers a new resolution tree in which the sets of crossings are ordered $$X_{1,3}<X_{2,3}<X_{1,2}<Y_{1,3}<Y_{2,3}<Y_{1,2}<Z$$ for $\Gamma_2$, and $$X_{1,2}<X_{2,3}<X_{1,3}<Y_{1,2}<Y_{2,3}<Y_{1,3}<Z$$ for $\Gamma_3$. Changing the ordering of the sets of crossings alters the resolution tree, but the sets $\Gamma_i$ for $i=1,\dots,4$ are invariant: they have the same leaves as the resolution tree $T_D$ defined at the beginning of this section.

Let us now study the contributions of $M=\Gamma_5\cup \Gamma_6\cup \Gamma_7$. We will compute each $\Gamma_i$ separately, and will add up the results.

In the case of $\Gamma_5$, we first index the crossings in the sets $X_{i,j}$ as follows: $X_{1,2}=\{c_1,\dots,c_p\}$, $X_{1,3}=\{c_{p+1},\dots,c_q\}$ and $X_{2,3}=\{c_{q+1},\dots,c_x\}$.

As we saw in the proof of \autoref{tricolor}, every path from the root of $T_D$ to a leaf in $\Gamma_5$ is determined by the following choices: two crossings $c_r\in X_{1,2}$ and $c_s\in X_{1,3}$, and the smoothings (not of type 2) that one applies to the crossings in $C$ (the set of all crossings of $D$). That is, each leaf of $\Gamma_5$ is determined by $r\in\{1,\ldots,p\}$, $s\in \{p+1,\ldots,q\}$, and a map $\varphi\in\{0,1\}^C$. We will denote such a leaf by $v_{\varphi,r,s}$, and we notice that its corresponding diagram is the monochromatic diagram $D_\varphi$ (recall notation from \autoref{notation}).

If we fix $c_r\in X_{1,2}$, $c_s\in X_{1,3}$ and $\varphi\in \{0,1\}^C$, we can describe in detail the path that leads from the root of $T_D$ to the leaf $v_{\varphi,r,s}$. It is obtained by the following sequence of smoothings:

\begin{enumerate}
    \item Apply smoothings of type 2 to the crossings $c_1,\dots,c_{r-1}\in X_{1,2}$.
    \item Smooth the crossing $c_r$ according to $\varphi$.
    \item Apply smoothings of type 2 to the crossings $c_{p+1},\dots,c_{s-1}\in X_{1,3}$.
    \item Smooth the crossing $c_s$ according to $\varphi$.
    \item Apply a $\overline{1-\varphi(c_i)}$ smoothing to each $c_i\in\{c_1,\dots,c_{r-1}\}\subset X_{1,2}$.
    \item Apply a $\overline{\varphi(c_j)}$ smoothing to each $c_j\in \{c_{r+1},\dots,c_{p}\}\subset X_{1,2}$.
    \item Apply a $\overline{1-\varphi(c_i)}$ smoothing to each $c_i\in\{c_{p+1},\dots,c_{s-1}\}\subset X_{1,3}$. 
    \item Apply a $\overline{\varphi(c_j)}$ smoothing to each $c_j\in \{c_{s+1},\dots,c_{q}\}\subset X_{1,3}$.
    \item Apply a $\overline{\varphi(c)}$ smoothing to every crossing $c\in X_{2,3}$.
    \item Apply a $\overline{\varphi(c)}$ smoothing to every crossing $c\in Y\cup Z$.
\end{enumerate}
If we denote $\sigma=\varphi_{|_{X_{1,2}}}$, $\tau=\varphi_{|_{X_{1,3}}}$, $\mu=\varphi_{|_{X_{2,3}}}$ and $\nu=\varphi_{|_{Y\cup Z}}$, we can consider the following tuples, which follow the notation in the proof of \autoref{teobicolor}:
\begin{itemize}
    \item $\sigma_{(r)}=(\overline{1-\varphi(c_1)},\dots,\overline{1-\varphi(c_{r-1})},\varphi(c_r), \overline{\varphi(c_{r+1})},\dots,\overline{\varphi(c_{p})})$,
        
    \item $\tau_{(s-p)}=(\overline{1-\varphi(c_{p+1})},\dots,\overline{1-\varphi(c_{s-1})},\varphi(c_s), \overline{\varphi(c_{s+1})},\dots,\overline{\varphi(c_{q})})$.
\end{itemize}
Notice that the smoothings (not of type 2) applied in the path from the root of $T_D$ to $v_{\varphi,r,s}$ are precisely those in $\sigma_{(r)}$ and $\tau_{(s-p)}$, followed by those determined by $\mu$ and $\nu$ (replacing $0$ by $\bar{0}$ and $1$ by $\bar{1}$, respectively). Hence, the polynomial $P_{\varphi,r,s}$ associated to the leaf $v_{\varphi,r,s}$ is:
\begin{equation*}
    P_{\varphi,r,s}=(-1)^{r-1}(A+A^{-1})(-1)^{s-p-1}(A+A^{-1})A^{e_{\sigma_{(r)}}}A^{e_{\tau_{(s-p)}}}A^{k_{\mu}}A^{k_{\nu}}\langle D_{\varphi}\rangle.
\end{equation*}
It follows that the polynomial $P_{\varphi}$ associated to all leaves of $\Gamma_5$ having diagram $D_{\varphi}$ is:

\begin{align*}
    P_{\varphi}&=\sum_{r=1}^p\sum_{s=p+1}^qP_{\varphi,r,s}\\
    &=\sum_{r=1}^p\sum_{s=p+1}^q (-1)^{r-1}(A+A^{-1})(-1)^{s-p-1}(A+A^{-1})A^{e_{\sigma_{(r)}}}A^{e_{\tau_{(s-p)}}}A^{k_{\mu}}A^{k_{\nu}}\langle D_{\varphi}\rangle\\
    &= \left[\sum_{r=1}^{x_{1,2}} (-1)^{r-1}(A+A^{-1})A^{e_{\sigma_{(r)}}}\right]\left[\sum_{t=1}^{x_{1,3}} (-1)^{t-1}(A+A^{-1})A^{e_{\tau_{(t)}}}\right]A^{k_{\mu}}A^{k_{\nu}}\langle D_{\varphi}\rangle
\end{align*}
The two sums in the last expression are analogous to those computed in the proof of \autoref{teobicolor}, so they are equal to $H_{k_{\sigma}}$ and $H_{k_{\tau}}$, respectively. Therefore,
$$
    P_{\varphi}= H_{k_\sigma}H_{k_\tau}A^{k_{\mu}}A^{k_{\nu}}\langle D_{\varphi}\rangle.
$$
Now we can add all these polynomials, taking into account that every map $\varphi\in \{0,1\}^{C}$ is determined by its restrictions to $X_{1,2}$, $X_{1,3}$, $X_{2,3}$ and $Y\cup Z$. Hence:
\begin{eqnarray*}
  \langle\langle D(\Gamma_5)\rangle\rangle & = & \sum_{\varphi\in \{0,1\}^C}{P_{\varphi}}
\\
 & = & \sum_{\sigma\in \{0,1\}^{X_{1,2}}}\sum_{\tau\in \{0,1\}^{X_{1,3}}}\sum_{\mu\in \{0,1\}^{X_{2,3}}}\sum_{\nu\in \{0,1\}^{Y\cup Z}}{H_{k_\sigma}H_{k_\tau}A^{k_{\mu}}A^{k_{\nu}}\langle D_{\varphi}\rangle}
\\
 & = & \sum_{\sigma\in \{0,1\}^{X_{1,2}}}\sum_{\tau\in \{0,1\}^{X_{1,3}}}\sum_{\mu\in \{0,1\}^{X_{2,3}}}{H_{k_\sigma}H_{k_\tau}A^{k_{\mu}}\left(\sum_{\nu\in \{0,1\}^{Y\cup Z}}{A^{k_{\nu}}\langle D_{\varphi}\rangle}\right)}
\\
 & = & \sum_{\sigma\in \{0,1\}^{X_{1,2}}}\sum_{\tau\in \{0,1\}^{X_{1,3}}}\sum_{\mu\in \{0,1\}^{X_{2,3}}}{H_{k_\sigma}H_{k_\tau}A^{k_{\mu}}\langle D_{\sigma\cup \tau\cup \mu}\rangle}
\end{eqnarray*}
where the last equality comes from the classical calculation of $\langle D_{\sigma\cup \tau\cup \mu}\rangle$, which is obtained by smoothing all crossings of $Y\cup Z$ in all possible ways.

We can do the sum along $\mu$ to obtain 
$$
\langle\langle D(\Gamma_5)\rangle\rangle = \displaystyle \sum_{\sigma\in \{0,1\}^{X_{1,2}}}\sum_{\tau\in \{0,1\}^{X_{1,3}}}{H_{k_\sigma}H_{k_\tau}\langle D_{\sigma\cup \tau}\rangle}.
$$
In order to add this expression to $\langle\langle D(\Gamma_6)\rangle\rangle$ and to $\langle\langle D(\Gamma_7)\rangle\rangle$, it is convenient to separate the above sum into more summands, separating the terms with distinct values of $\mu$ and also those terms with distinct values of $\phi= \varphi_{|_{Y_{2,3}}}$. We then obtain:
$$
\langle\langle D(\Gamma_5)\rangle\rangle = \displaystyle \sum_{\sigma\in \{0,1\}^{X_{1,2}}}\sum_{\tau\in \{0,1\}^{X_{1,3}}}\sum_{\mu\in \{0,1\}^{X_{2,3}}}\sum_{\phi\in \{0,1\}^{Y_{2,3}}}{H_{k_\sigma}H_{k_\tau}A^{k_\mu}A^{k_\phi}\langle D_{\sigma\cup\tau\cup\mu\cup\phi}\rangle}.
$$

This shows the case of $\Gamma_5$. Notice that, if $X_{1,2}=\varnothing$, there is only one possible $\sigma$ and $k_\sigma=0$, so $H_{k_\sigma}=0$, hence in this case $\langle\langle D(\Gamma_5)\rangle\rangle=0$. The same happens if $X_{1,3}=\varnothing$.

The case of $\Gamma_6$ is similar. The procedure to obtain a vertex $v_{\varphi,r,s}\in\Gamma_6$ from the root of $T_D$, choosing some $r\in \{1,\ldots,p\}$, some $s\in \{q+1,\ldots,x\}$, and some $\varphi\in \{0,1\}^C$ is the following:
\begin{enumerate}
    \item Apply smoothings of type 2 to the crossings $c_1,\dots,c_{r-1}\in X_{1,2}$.
    \item Smooth the crossing $c_r$ according to $\varphi$.
    \item Apply smoothings of type 2 to the crossings $c_{p+1},\dots,c_{q}\in X_{1,3}$.
    \item Apply smoothings of type 2 to the crossings $c_{q+1},\dots,c_{s-1}\in X_{2,3}$.
    \item Smooth the crossing $c_s$ according to $\varphi$.
    \item Apply a $\overline{1-\varphi(c_i)}$ smoothing to each $c_i\in\{c_1,\dots,c_{r-1}\}\subset X_{1,2}$.
    \item Apply a $\overline{\varphi(c_j)}$ smoothing to each $c_j\in \{c_{r+1},\dots,c_{p}\}\subset X_{1,2}$.
    \item Apply a $\overline{1-\varphi(c_i)}$ smoothing to each $c_i\in\{c_{p+1},\dots,c_{q}\}\subset X_{1,3}$. 
    \item Apply a $\overline{1-\varphi(c_i)}$ smoothing to each $c_i\in\{c_{q+1},\dots,c_{s-1}\}\subset X_{2,3}$. 
    \item Apply a $\overline{\varphi(c_j)}$ smoothing to each $c_j\in \{c_{s+1},\dots,c_{x}\}\subset X_{2,3}$.
    \item Apply a $\overline{\varphi(c)}$ smoothing to every crossing $c\in Y\cup Z$.
\end{enumerate}

As above, we define  $\sigma=\varphi_{|_{X_{1,2}}}$, $\tau=\varphi_{|_{X_{1,3}}}$, $\mu=\varphi_{|_{X_{2,3}}}$, $\nu=\varphi_{|_{Y\cup Z}}$, and the tuple $$\mu_{(s-q)}=(\overline{1-\varphi(c_{q+1})},\dots,\overline{1-\varphi(c_{s-1})},\varphi(c_s), \overline{\varphi(c_{s+1})},\dots,\overline{\varphi(c_{x})}).$$
Notice that the crossings in $X_{1,3}$ are smoothed in the opposite way it is indicated by $\tau$. In this case, the polynomial $P_{\varphi,r,s}$ associated to the leaf $v_{\varphi,r,s}$ is:
\begin{equation*}
    P_{\varphi,r,s}=(-1)^{r-1}(A+A^{-1})(-1)^{x_{1,3}}(-1)^{s-q-1}(A+A^{-1}) A^{e_{\sigma_{(r)}}}A^{e_{\mu_{(s-q)}}}A^{-k_{\tau}}A^{k_{\nu}}\langle D_{\varphi}\rangle.
\end{equation*}

If we sum all these polynomials for all $r$ and $s$, we obtain:
$$
 P_{\varphi}=\left[\sum_{r=1}^{x_{1,2}} (-1)^{r-1}(A+A^{-1})A^{e_{\sigma_{(r)}}}\right]\left[\sum_{t=1}^{x_{2,3}} (-1)^{t-1}(A+A^{-1})A^{e_{\mu_{(t)}}}\right](-1)^{x_{1,3}}A^{-k_{\tau}}A^{k_{\nu}}\langle D_{\varphi}\rangle
$$
$$
  = H_{k_\sigma}H_{k_\mu}(-1)^{x_{1,3}}A^{-k_{\tau}}A^{k_{\nu}}\langle D_{\varphi}\rangle 
  = H_{k_\sigma}H_{k_\mu}(-A)^{-k_{\tau}}A^{k_{\nu}}\langle D_{\varphi}\rangle,
$$
where the last equality holds since $x_{1,3}$ and $k_{\tau}$ have the same parity. Finally, adding all polynomials corresponding to vertices in $\Gamma_6$, we obtain:
\begin{eqnarray*}
  \langle\langle D(\Gamma_6)\rangle\rangle & = & \sum_{\varphi\in \{0,1\}^{C}}{P_\varphi} 
\\
  & = & \sum_{\sigma\in \{0,1\}^{X_{1,2}}} \sum_{\tau\in \{0,1\}^{X_{1,3}}} \sum_{\mu\in \{0,1\}^{X_{2,3}}} \sum_{\nu\in \{0,1\}^{Y\cup Z}}{H_{k_\sigma}(-A)^{-k_{\tau}}H_{k_\mu}A^{k_{\nu}}\langle D_{\varphi}\rangle}
\\
  & = & \sum_{\sigma\in \{0,1\}^{X_{1,2}}} \sum_{\tau\in \{0,1\}^{X_{1,3}}} \sum_{\mu\in \{0,1\}^{X_{2,3}}} H_{k_\sigma}(-A)^{-k_{\tau}}H_{k_\mu}\langle D_{\sigma\cup \tau\cup \mu}\rangle.
\end{eqnarray*}

As in the previous case, we can separate the summands corresponding to $\phi= \varphi_{|_{Y_{2,3}}}$ and we obtain:
$$
\langle\langle D(\Gamma_6)\rangle\rangle = \displaystyle \sum_{\sigma\in \{0,1\}^{X_{1,2}}}\sum_{\tau\in \{0,1\}^{X_{1,3}}}\sum_{\mu\in \{0,1\}^{X_{2,3}}}\sum_{\phi\in \{0,1\}^{Y_{2,3}}}{H_{k_\sigma}(-A)^{-k_{\tau}}H_{k_\mu}A^{k_\phi}\langle D_{\sigma\cup\tau\cup\mu\cup\phi}\rangle}.
$$

Now let us study $\Gamma_7$. The choices in this case are a crossing $c_r\in X_{1,3}$, a crossing $c_s\in Y_{2,3}$ (since legal crossings from $Y_{2,3}$ become illegal after the smoothing of $c_r$, which transforms color 3 into color 1),  and a map $\varphi\in \{0,1\}^C$. Suppose that the crossings in $Y_{2,3}$ are $\{c_{u+1},\ldots,c_v\}$. The steps are:
\begin{enumerate}
    \item Apply smoothings of type 2 to the crossings $c_{1},\dots,c_{p}\in X_{1,2}$.
    \item Apply smoothings of type 2 to the crossings $c_{p+1},\dots,c_{r-1}\in X_{1,3}$.
    \item Smooth the crossing $c_r$ according to $\varphi$.
    \item Apply smoothings of type 2 to the crossings $c_{u+1},\dots,c_{s-1}\in Y_{2,3}$.
    \item Smooth the crossing $c_s$ according to $\varphi$.
    \item Apply a $\overline{1-\varphi(c_i)}$ smoothing to each $c_i\in \{c_{1},\dots,c_{p}\}\subset X_{1,2}$. 
    \item Apply a $\overline{1-\varphi(c_i)}$ smoothing to each $c_i\in \{c_{p+1},\dots,c_{r-1}\}\subset X_{1,3}$.
    \item Apply a $\overline{\varphi(c_j)}$ smoothing to each $c_i\in \{c_{r+1},\dots,c_{q}\}\subset X_{1,3}$.
    \item Apply a $\overline{\varphi(c)}$ smoothing to every crossing $c\in X_{2,3}\cup Y_{1,2}\cup Y_{1,3}$.
    \item Apply a $\overline{1-\varphi(c_i)}$ smoothing to each $c_i\in \{c_{u+1},\dots,c_{s-1}\}\subset Y_{2,3}$.
    \item Apply a $\overline{\varphi(c_j)}$ smoothing to each $c_i\in \{c_{s+1},\dots,c_{v}\}\subset Y_{2,3}$.
    \item Apply a $\overline{\varphi(c)}$ smoothing to every crossing $c\in Z$.
\end{enumerate}

We set $\sigma=\varphi_{|_{X_{1,2}}}$, $\tau=\varphi_{|_{X_{1,3}}}$, $\mu=\varphi_{|_{X_{2,3}}}$, $\phi=\varphi_{|_{Y_{2,3}}}$, $\psi=\varphi_{|_{Y_{1,2}\cup Y_{1,3}\cup Z}}$, and the tuple $$\phi_{(s-u)} =(\overline{1-\varphi(c_{u+1})},\dots,\overline{1-\varphi(c_{s-1})},\varphi(c_s), \overline{\varphi(c_{s+1})},\dots,\overline{\varphi(c_{v})}).$$ Then we have:
\begin{equation*}
    P_{\varphi,r,s}=(-1)^{x_{1,2}}(-1)^{r-p-1}(A+A^{-1})(-1)^{s-u-1}(A+A^{-1}) A^{-k_{\sigma}}A^{e_{\tau_{(r-p)}}}A^{k_\mu}A^{e_{\phi_{(s-u)}}}A^{k_{\psi}}\langle D_{\varphi}\rangle,
\end{equation*}

giving rise to 
\begin{equation*}
    P_{\varphi}=(-A)^{-k_{\sigma}} H_{k_\tau} A^{k_\mu} A
    ^{k_\psi}  H_{k_\phi} \langle D_{\varphi}\rangle.
\end{equation*}
The sum is:
$$
  \langle\langle D(\Gamma_7)\rangle\rangle 
   = \sum_{\sigma\in \{0,1\}^{X_{1,2}}} \sum_{\tau\in \{0,1\}^{X_{1,3}}} \sum_{\mu\in \{0,1\}^{X_{2,3}}} \sum_{\phi\in \{0,1\}^{Y_{2,3}}} (-A)^{-k_{\sigma}}H_{k_\tau}A^{k_\mu}H_{k_\phi} \langle D_{\sigma\cup \tau\cup \mu\cup\phi}\rangle.
$$

Now we can add the three expressions, which we have written as follows:
$$
  \langle\langle D(\Gamma_5)\rangle\rangle 
   = \sum_{\sigma\in \{0,1\}^{X_{1,2}}} \sum_{\tau\in \{0,1\}^{X_{1,3}}} \sum_{\mu\in \{0,1\}^{X_{2,3}}} \sum_{\phi\in \{0,1\}^{Y_{2,3}}} H_{k_\sigma}H_{k_\tau}A^{k_{\mu}} A^{k_\phi}\langle D_{\sigma\cup \tau\cup \mu\cup \phi}\rangle.
$$
$$
  \langle\langle D(\Gamma_6)\rangle\rangle 
   = \sum_{\sigma\in \{0,1\}^{X_{1,2}}} \sum_{\tau\in \{0,1\}^{X_{1,3}}} \sum_{\mu\in \{0,1\}^{X_{2,3}}} \sum_{\phi\in \{0,1\}^{Y_{2,3}}} H_{k_\sigma}(-A)^{-k_{\tau}}H_{k_\mu}A^{k_\phi}\langle D_{\sigma\cup \tau\cup \mu\cup\phi}\rangle.
$$
$$
  \langle\langle D(\Gamma_7)\rangle\rangle 
   = \sum_{\sigma\in \{0,1\}^{X_{1,2}}} \sum_{\tau\in \{0,1\}^{X_{1,3}}} \sum_{\mu\in \{0,1\}^{X_{2,3}}} \sum_{\phi\in \{0,1\}^{Y_{2,3}}} (-A)^{-k_{\sigma}}H_{k_\tau}A^{k_\mu}H_{k_\phi} \langle D_{\sigma\cup \tau\cup \mu\cup\phi}\rangle.
$$

We then need to show what is the result of adding up
$$
H_aH_bA^cA^d + H_a(-A)^{-b}H_cA^d + (-A)^{-a}H_bA^cH_d
$$
for given integers $a,b,c,d$. We will call this sum $S$.

We first notice that $H_k= A^{k}+(-1)^{k+1}A^{-k} = A^k-(-A)^{-k}$. Therefore, we have:
\begin{align*}
   H_aH_bA^cA^d & = & & A^{a} A^{b} A^{c} A^{d}  & - & (-A)^{-a} A^{b} A^{c} A^{d} 
\\ 
                &  & - & A^{a} (-A)^{-b} A^{c} A^{d} & + & (-A)^{-a} (-A)^{-b} A^{c} A^{d}
\\
   H_a(-A)^{-b}H_cA^d & = & & A^{a} (-A)^{-b} A^{c} A^{d} & - & (-A)^{-a} (-A)^{-b} A^{c} A^{d} 
\\
                & & - & A^{a} (-A)^{-b} (-A)^{-c} A^{d} & + & (-A)^{-a} (-A)^{-b} (-A)^{-c} A^{d}
\\
   (-A)^{-a}H_bA^cH_d & = & & (-A)^{-a} A^{b} A^{c} A^{d} & - & (-A)^{-a} (-A)^{-b} A^{c} A^{d} 
\\
                & & - & (-A)^{-a} A^{b} A^{c} (-A)^{-d} & + & (-A)^{-a} (-A)^{-b} A^{c} (-A)^{-d}
\end{align*}
If we numerate from (1) to (12) the summands on the right side of the equalities that we want to add up, in the order they appear in the above expressions and all  with positive sign, the total sum is
$$
  S =  (1)-(2)-(3)+(4)+(5)-(6)-(7)+(8)+(9)-(10)-(11)+(12).
$$ 
We see that (2)=(9), (3)=(5), and (4)=(6). Therefore, these summands cancel and the total sum is:
$$
  S =  (1)-(7)+(8)-(10)-(11)+(12).
$$
We now add an extra summand $(13)=(-A)^{-a}(-A)^{-b}(-A)^{-c}(-A)^{-d}$ and its opposite, and we reorder the summands, so we obtain:
$$
   S =  [(1)-(7)-(11)+(13)]-[(10)-(8)-(12)+(13)].
$$
Finally, we simplify the expressions in each bracket, say $S_1$ and $S_2$. On the one hand, we have:
\begin{eqnarray*}
S_1 & = &  (1)-(7)-(11)+(13) 
\\
   & = & A^{a+d}A^{b+c} - A^{a+d}(-A)^{-b-c} - (-A)^{-a-d}A^{b+c} + (-A)^{-a-d}(-A)^{-b-c}
\\
   & = & \left( A^{a+d}- (-A)^{-a-d}\right)\left( A^{b+c}- (-A)^{-b-c}\right)
\\
   & = & H_{a+d}H_{b+c}.
\end{eqnarray*}
On the other hand:
\begin{eqnarray*}
S_2 & = &  (10)-(8)-(12)+(13) 
\\
   & = & (-A)^{-a-b}\left[A^{c}A^{d} - (-A)^{-c}A^{d} -A^{c}(-A)^{-d} + (-A)^{-c}(-A)^{-d}\right]
\\   
   & = & (-A)^{-a-b}\left[\left(A^{c}-(-A)^{-c}\right)\left(A^{d}-(-A)^{-d}\right)\right]
\\
   & = & (-A)^{-a-b} H_{c}H_{d}.
\end{eqnarray*}
Therefore, we finally obtain:
$$
S = S_1 - S_2 = H_{a+d}H_{b+c} - (-A)^{-a-b} H_{c}H_{d}.
$$
Or, alternatively,
$$
S = H_{a+d}H_{b+c} + (-1)^{a+b-1}A^{-a-b} H_{c}H_{d}.
$$
The final value of $\langle\langle D(M)\rangle\rangle$ is the sum of all expressions $S$ for distinct values of $a=k_{\sigma}$, $b=k_{\tau}$, $c=k_{\mu}$, and $d=k_{\phi}$, hence we obtain the expression in the statement.
\end{proof}

\vspace{0.1cm}

\bibliographystyle{amsplain}
\bibliography{biblio}{}

\providecommand{\bysame}{\leavevmode\hbox to3em{\hrulefill}\thinspace}
\providecommand{\MR}{\relax\ifhmode\unskip\space\fi MR }
\providecommand{\MRhref}[2]{%
  \href{http://www.ams.org/mathscinet-getitem?mr=#1}{#2}
}
\providecommand{\href}[2]{#2}
\begin{thebibliography}{1}

\bibitem{Aicardi2016}
F.~Aicardi and J.~Juyumaya, \emph{{Tied links}}, Journal of Knot Theory and its Ramifications \textbf{25} (2016), no.~9, 1--28.

\bibitem{Aicardi2018}
\bysame, \emph{{Kauffman type invariants for tied links}}, Mathematische Zeitschrift \textbf{289} (2018), no.~1-2, 567--591.

\bibitem{Aicardi2020}
\bysame, \emph{{Two parameters bt-algebra and invariants for links and tied links}}, Arnold Mathematical Journal \textbf{6} (2020), no.~1, 131--148.

\bibitem{Aicardi2021}
\bysame, \emph{Tied links and invariants for singular links}, Advances in Mathematics \textbf{381} (2021), 107629.

\bibitem{Chlouveraki2020}
M.~Chlouveraki, J.~Juyumaya, K.~Karvounis, and S.~Lambropoulou, \emph{{Identifying the Invariants for Classical Knots and Links from the Yokonuma-Hecke Algebras}}, International Mathematics Research Notices \textbf{2020} (2020), no.~1, 214--286.

\bibitem{Cardenas2024}
O.~Cárdenas-Andaur, \emph{On the {A}icardi-{J}uyumaya bracket for tied links}, Journal of Knot Theory and its Ramifications (To appear).

\bibitem{Kauffman1987}
L.~H. Kauffman, \emph{{State models and the Jones polynomial}}, Topology \textbf{26} (1987), no.~3, 395--407.

\bibitem{Khovanov2000}
M.~Khovanov, \emph{A categorification of the {J}ones polynomial}, Duke Mathematical Journal \textbf{101} (2000), 359--426.

\end{thebibliography}

\end{document}